\documentclass[11pt]{amsart}
\usepackage{a4}
\usepackage{amssymb}
\usepackage{stackengine}

\stackMath

\usepackage{color}
\makeatletter
\newcommand{\leqnomode}{\tagsleft@true}
\newcommand{\reqnomode}{\tagsleft@false}
\makeatother

\newtheorem{theorem}{Theorem}
\theoremstyle{plain}
\newtheorem{lemma}[theorem]{Lemma}
\newtheorem{proposition}[theorem]{Proposition}

\newtheorem{corollary}[theorem]{Corollary}

\newtheorem{definition}[theorem]{Definition}

\DeclareMathOperator*{\Sing }{Sing}
\DeclareMathOperator*{\esssup}{ess\,sup}

\DeclareMathOperator*{\essinf}{ess\,inf}

\newcommand\trnorm[1]{\left\|\kern-1.2pt\left|#1\right\|\kern-1.2pt\right|}

\newcommand{\N}{\mathbb{N}}
\newcommand{\C}{\mathbb{C}}
\newcommand{\R}{\mathbb{R}}
\DeclareMathOperator*{\bd}{bd}
\DeclareMathOperator*{\dist}{dist}
\DeclareMathOperator*{\supp}{supp}
\DeclareMathOperator*{\esssupp}{ess\,supp}

\let\latexchi\chi
\makeatletter
\renewcommand\chi{\@ifnextchar_\sub@chi\latexchi}
\newcommand{\sub@chi}[2]{
  \@ifnextchar^{\subsup@chi{#2}}{\latexchi^{}_{#2}}%
}
\newcommand{\subsup@chi}[3]{
  \latexchi_{#1}^{#3}%
}
\makeatother
\begin{document}

\title{Density of smooth functions in Musielak-Orlicz spaces}



\author{Anna Kami\'{n}ska}
\address{Department of Mathematical Sciences,
The University of Memphis, TN 38152-3240}
\email{kaminska@memphis.edu}

\author{Mariusz \.Zyluk}
\address{Department of Mathematical Sciences,
The University of Memphis, TN 38152-3240}
\email{mzyluk@gmail.com}

\date{\today}

\thanks{}

\begin{abstract} We provide necessary and sufficient conditions for the space of smooth functions with compact supports $C^\infty_C(\Omega)$ to be dense in Musielak-Orlicz spaces $L^\Phi(\Omega)$ where $\Omega$ is an open subset of $\mathbb{R}^d$. In particular we prove that if $\Phi$ satisfies condition $\Delta_2$, the closure of $C^\infty_C(\Omega)\cap L^\Phi(\Omega)$ is equal to $L^\Phi(\Omega)$ if and only if the measure of singular points of $\Phi$ is equal to zero. This  extends the earlier density theorems proved under the assumption of local integrability of $\Phi$, which implies that the measure of the singular points of $\Phi$ is zero.  As a corollary we obtain  analogous results for Musielak-Orlicz spaces generated by double phase functional and we recover the well known result  for  variable exponent Lebesgue spaces. 
\end{abstract}

\maketitle

 We study here the problem of density of the space $C_C^\infty(\Omega)$  of smooth functions with compact supports in a subspace of order continuous functions $E^\Phi(\Omega)$ of Musielak-Orlicz spaces $L^\Phi(\Omega)$ over an open set  $\Omega\subset \mathbb{R}^d$. This is a standard problem in function spaces and it  has been considered before in $L^\Phi(\Omega)$ under some restrictions on the function $\Phi$. For rearrangement invariant spaces, and in particular  Orlicz spaces, it is well known and standard to prove that smooth functions are dense in the subspace of order continuous elements, which is equivalent to that simple functions are dense in this subspace. In Musielak-Orlicz spaces the situation is different. First we notice that not every simple function belongs to $E^\Phi(\Omega)$. However it is possible to show that the set of simple functions in $E^\Phi(\Omega)$ is big enough to be dense in this space. On the other hand  the case of smooth functions is different. There exist functions $\Phi$ such that the smooth functions are not dense in $E^\Phi(\Omega)$.  In fact we  will present  here necessary and sufficient condition for $\Phi$ in order to a subspace of the space of smooth functions $C_C^\infty(\Omega)$ is dense in $E^\Phi(\Omega)$.
 
In the recent decade  the Musielak-Orlicz spaces and their particular examples  of the variable exponent Lebesgue spaces or generated by double phase functionals have gained special attention in the context of studies the solutions of PDE belonging to the spaces \cite{CUF, DHHR, HHK, HH, YA}.   The knowledge when  smooth functions with compact supports are dense in the space is basic for this research.  

Let $\mathbb{R}^d$ be the $d$-dimensional Euclidean space equipped with the Lebesgue measure $|\cdot|$.
Given a Lebesgue measurable set $A\subset \R^d$, the set of all Lebesgue measurable complex valued functions on $A$ will be denoted as $L^0(A)$. For  $f\in L^0(A)$ its support is defined as the set 
\[
\supp(f)=\{x\in A: f(x)\neq 0\}.
\]
The set of all simple, complex  valued functions on $A$ is denoted as
  \[S(A)=\{ f\in L^0(A): \ |\supp f|<\infty \text{ and } f \text{ has finitely many values}\}.\] 
 For any open set $\Omega\subset \R^d$ and any $f\in L^0(\Omega)$ we define the essential support of $f$ as 
\[\esssupp (f)=\Omega\setminus\bigcup \{U\subset \Omega: U\text{ is open and } f=0 \text{ a.e. on } U\}.\]
Notice that $\esssupp (f)$ is a closed subset of $\Omega$.


Throughout this paper  $\Omega\subset \R^d$ stands always for an open set. Recall a function $f: \Omega \mapsto \C$ is said to be smooth if it possess all derivatives, and the set of all smooth functions is  denoted by $C^\infty(\Omega)$.  Notice also that, for $f\in C^\infty(\Omega)$,
\[\esssupp (f)=\overline{\supp(f)}.\]

 By $C_C^\infty(\Omega)$ we denote the set of all smooth compactly supported functions defined on $\Omega$, that is
\[
C_C^\infty(\Omega)=\{f\in C^\infty(\Omega): \ \esssupp(f)\text{ is compact}\}.
\]

 For any  point $(x_1,\dots,x_d)=x\in\R^d$ and $l>0$ we define the open cube of center $x$ and side-length $l$ as
\[
Q(x,l)=\left\{(y_1,\dots,y_d)\in\R^d: \ |x_n-y_n|<\frac{l}{2}, \text{ for }n=1,\dots,d \right\}.
\]
Similarly, for any  point  $x\in \R^d$ and any $r>0$ we define the open ball of center $x$ and radius $r$ as
\[B(x,r)=\{y\in\R^d: \ |x-y|<r\}.\]
A general open ball or cube often will be given without specifying their centers and radii or side lengths. In those cases an open ball will be usually denoted by $B$ and an open cube as $Q$.

\begin{definition}\label{def}
Let $A\subset \R^d$ be  Lebesgue measurable.  A function $\Phi:A\times [0,\infty )\to [0,\infty )$ is called a Musielak-Orlicz function  {\rm (}MO function{\rm )} on $A$  if
\begin{enumerate}
\item[$(a)$]  for every $x\in A$, $t\mapsto\Phi(x,t)$ is convex, 
\item[$(b)$] for every $x\in A$, $\Phi(x,t)=0$ if and only if $t=0$, 
\item[$(c)$] for every $ t\in [0,\infty) ,\ x\mapsto\Phi (x,t)$ is  Lebesgue measurable on $A$.
\end{enumerate}
\end{definition}

 A $MO$ function $\Phi$ is said to satisfy the $\Delta_2$ condition if there exist a constant $C>0$ and a positive function $h\in L^1(A)$, that is $\int_A h(x)\, dx <\infty$, such that 
\[
 \   \Phi(x,2t)\leq C\Phi(x,t)+h(x), \ \ \ \ t\geq 0,\ \text{ a.a.} \ x\in  A.
 \]
Given a measurable set  $A\subset\R^d$ and $MO$ function $\Phi$, define the functional $I_\Phi$ on $L^0(A)$ by
\[
I_\Phi(f)=\int\limits_{A}\Phi(x,|f(x)|)dx.
\]
  The {\it Musielak-Orlicz space} ($MO$ space) $L^\Phi(A)$ is defined as
    \[L^\Phi(A)=\left\lbrace f\in L^0(A): \exists \lambda>0 \ I_\Phi (\lambda f)<\infty \right\rbrace,\]
and  its {\it subspace of finite elements} as   
\[E^\Phi (A)=\left\lbrace f\in L^0(A): \forall \lambda>0 \ I_\Phi (\lambda f)<\infty \right\rbrace.\]
 The functional
\[\|f\|_{\Phi}=\inf\left\lbrace \lambda>0: I_\Phi\left(f/\lambda\right)\leq1\right\rbrace, \ \ \ \   f\in L^0(A),
\]
is a norm on the space $L^\Phi(A)$, called the Luxemburg norm. For extensive  information on Musielak-Orlicz spaces reader is sent to \cite{Chen, CUF, DHHR, HH, HH1979, Kam1998, KamKub,  Kamzyl, Mus, zyl}. Recall the following auxiliary facts on $MO$ spaces.

\begin{theorem}\cite{Mus, zyl}\label{podstawowe wlasnosci}
Let $A$ be a Lebesgue measurable subset of $\R^d$ and $\Phi$ a $MO$ function on $A$. The following statements hold.
\begin{enumerate}
    \item[$(1)$] A sequence $\lbrace f_n\rbrace_{n=1}^\infty$  is convergent in norm  to $f$ in $L^\Phi(A)$, if and only if for every $\lambda>0$, $\lim\limits_{n\to\infty} I_\Phi(\lambda (f-f_n))=0$.    
    \item[$(2)$] The MO space $L^\Phi(A)$ equipped with the Luxemburg norm $\| \cdot\| _\Phi$ is a Banach function space and $E^\Phi(A)$ is a closed subspace of $L^\Phi(A)$. 
    \item[$(3)$] If $\{f_n\}_{n=1}^\infty\subset L^\Phi(A)$ converges to $0$ in the norm, then it converges to $0$ in measure on sets of finite measure.
 \item[$(4)$] $L^\Phi(A)=E^\Phi (A)$ if and only if 
    $\Phi$ satisfies $\Delta_2$ condition.
 \item[$(5)$]  The set $S(A)\cap E^\Phi(A)$ is dense in $(E^\Phi(A),\|\cdot\|_\Phi)$.  
    \end{enumerate}
\end{theorem}

It is well known that the space of smooth functions with compact supports is dense in the Lebesgue and Orlicz spaces on open subset of $\mathbb{R}^d$ \cite{Adams, Gossez}. In the case of Musielak-Orlicz spaces $L^\Phi(\Omega)$ there exist similar results but under some restrictions on  $\Phi$. Recall that  $\Phi$ is said to be locally integrable whenever 
\[
\int_K \Phi(x,t)\,dx < \infty
\]
 for each $t\ge 0$ and every compact set $K\subset\Omega$.  Notice that under the above assumption, $C_C^\infty (\Omega) \subset L^\Phi(\Omega)$. In \cite{HHK,YA} it has been proved in this case  that $C_C^\infty (\Omega)$ is dense in $E^\Phi(\Omega)$.
 
 Our main goal here is to prove the similar result, namely density of $C_C^\infty (\Omega) \cap E^\Phi(\Omega)$ in $E^\Phi(\Omega)$ for Musielak-Orlicz functions $\Phi$ without additional assumptions, in particular without local integrability of $\Phi$. Observe that in general $C_C^\infty(\Omega)$ is not contained in $E^\Phi(\Omega)$. Applying  Lemma \ref{Kaminska} we see that we have enough simple functions belonging to $E^\Phi(\Omega)$ in order to get condition (5) of Theorem \ref{podstawowe wlasnosci}. However for density of smooth functions we will need additional assumption that $|\Sing \Phi| = 0$.
  
  A simple example of the $MO$ function not locally integrable  is the function $\Phi_1:\R\times [0,\infty)\to [0,\infty)$ given  by the formula
\[\Phi_1(x,t)=\begin{cases} \frac{t}{|x|} & t\geq 0, \ x\neq 0  \\ 
0 & t\geq 0, \ x=0.\end{cases}\]
Clearly it is not locally integrable on any compact set $K$ consisting of the point $0$.
 
 A slightly more involved function that is not locally integrable can be constructed as follows. Let $\{r_n\}_{n=1}^\infty$ be an enumeration of rational numbers from the interval $(0,1)$. For any natural number $n$ and any $x\in (0,1)$ we set
 \[g_n(x)=\begin{cases} \frac{1}{x-r_n} &    x> r_n \\
0 &  x\leq r_n.\end{cases}\]
Notice that for any $n\in \N$ the function $g_n$ is not integrable on any interval containing $r_n$, but the function $\sqrt{g_n}$ is an element of $L^1(0,1)$ and $\| \sqrt{g_n}\|_1 = \int_0^1 \sqrt{g_n (x)} \,dx \leq \frac{2}{3}$ for every $n\in\mathbb{N}$. Hence $g(x)=\sum\limits_{n=1}^\infty\frac{\sqrt{g_n(x)}}{2^n}$ is an element of $L^1(0,1)$. Therefore there exists a subset $B$ of $(0,1)$ of measure $0$ such that for any $x\in (0,1)\setminus B$, $g(x)$ is finite.
Notice now that for any $x\in (0,1)\setminus B$ we have that 
\[w(x): =\sum\limits_{n=1}^\infty\frac{g_n(x)}{4^n}\leq g(x)^2<\infty.\]
We define now the function $\Phi_2: \R\times [0,\infty)\to [0,\infty)$ by the formula 
\[\Phi_2(x,t)=\begin{cases}t w(x) &  \  x\in (0,1)\setminus B, \ t\geq 0\\
0 & x\in B, \ t \geq 0 \\
t & x\in \R\setminus (0,1), \  t\geq 0 .\end{cases},\]
By construction, the function $\Phi_2$ is not  locally integrable $MO$ function. In fact $w$ is not integrable on any open subinterval of $(0,1)$.
It is clear that both $\Phi_1$ and $\Phi_2$ satisfy the $\Delta_2$ condition. Hence by Theorem \ref{podstawowe wlasnosci} (5), simple functions that belong to $L^{\Phi_i} (\R)$ are dense in $L^{\Phi_i} (\R)$, for $i=1,2$.  But what about density of compactly supported smooth functions belonging to the space? It turns out that they are dense in  $L^{\Phi_1} (\R)$ but not in $L^{\Phi_2} (\R)$. The discussion below explains why that is the case and characterizes those $MO$ functions  for which elements of $C_C^\infty (\Omega)\cap E^\Phi(\Omega)$ are dense in $E^\Phi(\Omega)$.

 We start with the following definition.
\begin{definition}
Let $\Omega\subset \R^d$ be an open set and $\Phi$ be a MO function defined on $\Omega$. We define the set $\Sing \Phi$, called the set of singular points of $\Phi$, as
\[ \Sing \Phi= \lbrace x\in \R^d:\ \forall r>0 \ \exists t_r>0  \int\limits_{B(x,r)\cap \Omega}\Phi(y,t_r)dy=\infty\rbrace.\] 
\end{definition}

Clearly,  if $\Phi$ on  $\Omega\subset\R^d$ is locally integrable then $|\Sing \Phi| =0$. However the converse implication is not satisfied in view of $|\Sing \Phi_1|=0$ and $\Phi_1$ being not locally integrable. 

We also notice that  for $\Phi_1$ and $\Phi_2$  we have
\[\Sing \Phi_1= \{0\},\]
\[\Sing \Phi_2= [0,1],\]
which implies  that both of those sets are closed. This is not a coincidence and the following proposition holds.
\begin{proposition}\label{dom}
For any open set $\Omega\subset \R^d$ and MO  function $\Phi$ defined on $\Omega$, the set $\Sing \Phi$ is a closed subset of $\R^d$.
\end{proposition}

\begin{proof}
Let $x\in \R^d$, $\lbrace x_n\rbrace_{n=1}^\infty\subset  \text{Sing } \Phi$ and $\lim\limits_{n\to\infty}x_n=x$. We will show that $x\in \text{Sing } \Phi$. For any $r>0$, then there exists $N\in\N$ such that for all $n> N$, $|x-x_n|<\frac{r}{2}$. Notice that, for each $n>N$ we have $B(x_n,\frac{r}{2})\subset B(x,r)$. 
For any $n>N$, by definition of $\text{Sing } \Phi$ there exists $t_{\frac{r}{2}}>0$ such that
\[ \int\limits_{B(x_n,\frac{r}{2})\cap\Omega}\Phi(y,t_{\frac{r}{2}})dy=\infty,\]
therefore
\[ \int\limits_{B(x,r)\cap\Omega}\Phi(y,t_{\frac{r}{2}})dy\geq\int\limits_{B(x_n,\frac{r}{2})\cap\Omega}\Phi(y,t_{\frac{r}{2}})dy=\infty.\]
Since $r$ was arbitrary chosen, we conclude that $x\in \text{Sing }\Phi$. 

\end{proof}

 Now we show that for density of $C_C^\infty (\Omega)\cap E^\Phi(\Omega)$ in $E^\Phi(\Omega)$, the set $\Sing\Phi$ has to be of zero measure. We will need the following result.
 
 \begin{lemma}\cite[p. 64]{zyl, Kam1984}\label{Kaminska}
Let $A\subset \R^d$ be measurable and $\Phi$ be a MO function on $A$. There exists a sequence of pairwise disjoint set $\lbrace A_n\rbrace_{n=1}^\infty\subset A$ such that   for each $n\in \N, |A_n|<\infty$, $\left|A\setminus\bigcup\limits_{n=1}^\infty A_n\right|=0$,  and 
\[
\sup\limits_{x\in A_n}\Phi(x,t)<\infty,\ \ \ \ n\in \N, \ \ t\geq 0.
\]
Consequently $L^\infty(A|_{A_n}) \subset E^\Phi(A|_{A_n})$ for every $n\in\N$.
\end{lemma}

\begin{theorem}\label{koniecznosc sing}

 Let $\Omega\subset \R^d$ be open and $\Phi$ be a MO function on $\Omega$.  If $|\Sing\Phi|>0$, then $C_C^\infty(\Omega)\cap E^\Phi (\Omega)$ is not dense in $E^\Phi(\Omega)$.
\end{theorem}
\begin{proof}
We argue by contradiction. Assume that $C_C^\infty(\Omega)\cap E^\Phi (\Omega)$ is dense in $E^\Phi(\Omega)$ and $|\text{Sing }\Phi|>0$. Let $\lbrace \Omega_n \rbrace_{n=1}^\infty$ be a sequence of sets such that $\Omega_n=A_n$, where $\{A_n\}_{n=1}^\infty$ is the sequence of sets from the conclusion of Lemma \ref{Kaminska}. Since $\left|\Omega\setminus\bigcup\limits_{n=1}^\infty \Omega_n\right|=0$, there exists $n_0$ such that $|\Sing \Phi\,\cap \,\Omega_{n_0}|>0$. By inner regularity of the Lebesgue measure there exists a compact set $K\subset \Sing \Phi\cap \Omega_{n_0}$ such that $0<|K|<|\Sing \Phi\cap \Omega_{n_0}|$. For any $t>0,$
\[\int\limits_{K} \Phi(x,t)dx\leq\int\limits_{\Omega_{n_0}}\Phi(x,t)dx<\infty,\]
therefore $\chi_{K}\in E
^\Phi(\Omega)$. By density of $C_C^\infty(\Omega)\cap E^\Phi (\Omega)$ in $E^\Phi(\Omega)$ there exists a sequence of functions $f_n\in C_C^\infty(\Omega)\cap E^\Phi (\Omega)$ such that, 
\[\lim\limits_{n\to\infty} \| \chi_K - f_n\| _\Phi=0.\]
Therefore, by Theorem \ref{podstawowe wlasnosci} (3), $\{\chi_K - f_n\}_{n=1}^\infty$ converges to $0$ in measure for any set $A\subset\Omega$ with $|A|<\infty$.
 It follows that there exists a subsequence $\lbrace n_k \rbrace_{k=1}^\infty\subset \N$ such that $\lim\limits_{k\to\infty}\chi_K(x) - f_{n_k}(x)=0$ for a.a. $ x\in\Omega$. For convenience  let us rewrite $f_{n_k}$ as $f_k$. Since $|K|>0$ there exists $x_0\in K$ such that $\lim\limits_{k\to\infty} f_k(x_0)=1$. 
Take now $k_0$ such that 
\[\frac{1}{2}<f_{k_0}(x_0)<\frac{3}{2}.\]
By continuity of $f_{k_0}$, there exists a ball $B(x_0,r)$, such that for any $y\in B(x_0,r)$,
\[-\frac{1}{4}<f_{k_0}(y)-f_{k_0}(x_0)<\frac{1}{4}.\]
Combining both of the above inequalities, we get that for any $y\in B(x_0,r)$,
\[\frac{1}{4}<|f_{k_0}(y)|<\frac{7}{4},\]
therefore $\frac{1}{4}\chi_{B(x_0,r)}<|f_{k_0}|$, so $\chi_{B(x_0,r)}\in E^\Phi(\Omega)$. Hence, for each $\lambda>0$ we have
\[\int\limits_{B(x_0,r)}\Phi(y,\lambda)dy<\infty.\]
On the other hand, $x_0\in K\subset \Sing\Omega$, and therefore there exists a $t_{r}>0$ such that
\[\int\limits_{B(x_0,r)}\Phi(y,t_r)=\infty,\]
and so $\chi_{B(x_0,r)}\notin E^\Phi(\Omega)$, a contradiction.
\end{proof}

To prove that the condition $|\Sing\Phi|=0$ is sufficient for density of $C_C^\infty(\Omega)$ functions in $E^\Phi(\Omega)$ we need the following fact about compact sets. 

\begin{lemma}\label{nested open sets}
Let $K$ be a compact subset of $\R^d$. There exists a sequence of open, bounded sets $\lbrace U_n\rbrace_{n=1}^\infty$ such that, for each $n\in \N$,
\begin{enumerate}
    \item[$(1)$]   $K\subset U_n$,
    \item[$(2)$]  $|\overline{U_n}|=|U_n|<|K|+\frac{1}{n}, $
    \item[$(3)$]  $\overline {U_{n+1}}\subset U_n$, $|U_{n+1}|<|U_n|.$
\end{enumerate}
\end{lemma}
\begin{proof}
Let $K\subset \R^d$ be compact. We will construct the sequence $\lbrace U_n\rbrace_{n=1}^{\infty}$ by induction.

 Let $\lbrace Q_{1,n}\rbrace_{n=1}^\infty$ be a sequence of open cubes such that $\left|\bigcup\limits_{n=1}^\infty Q_{1,n}\right|<|K|+1$ and $K\subset\bigcup\limits_{n=1}^\infty Q_{1,n}$.
By compactness of $K$, there exists $M_1$ cubes, $\lbrace Q_{1,n}\rbrace_{n=1}^{M_1}$, such that  $K\subset\bigcup\limits_{n=1}^{M_1} Q_{1,n}$.
Notice that $\bigcup\limits_{n=1}^{M_1} Q_{1,n}$ can be written as a finite union of disjoint cuboids $\lbrace R_j\rbrace_{j=1}^{J_1}$ (a cuboid is a finite intersection of open cubes),
\[\bigcup\limits_{n=1}^{M_1} Q_{1,n}=\bigcup\limits_{j=1}^{J_1} R_j.\]
Moreover for each $1\leq j\leq J_1$, $|\overline{R_j}|=|R_j|$. Therefore
\[\left|\overline{\bigcup\limits_{n=1}^{M_1} Q_{1,n}}\right|=\left|\overline{\bigcup\limits_{j=1}^{J_1} R_j}\right|= \left|\bigcup\limits_{j=1}^{J_1}\overline{R_j}\right|\leq \sum\limits_{j=1}^{J_1}\left|\overline{R_j}\right|=\sum\limits_{j=1}^{J_1}\left|R_j\right|=\left|\bigcup\limits_{j=1}^{J_1} R_j\right|=\left|\bigcup\limits_{n=1}^{M_1} Q_{1,n}\right|,\]
so 
\[\left|\overline{\bigcup\limits_{n=1}^{M_1} Q_{1,n}}\right|=\left|\bigcup\limits_{n=1}^{M_1} Q_{1,n}\right|.\]
Defining $U_1=\bigcup\limits_{n=1}^{M_1} Q_{1,n}$ we have $K\subset U_1$, $|\overline{U_1}|=|U_1|$ and $|U_1|< |K|+1$.

 Assume that the set $U_n$ is constructed. Now we will construct the set $U_{n+1}$. First notice  that $\overline{U_n}$ is compact, therefore the boundary of $U_n$, $\bd (U_n)$ is also compact.

 Let $f(x)=\dist (x,\bd (U_n))$ be the distance of $x$ from $\bd(U_n)$. The function $f$ is continuous on $\R^d$ and for any $x\in K$, $f(x)>0$. By compactness of  $K$, there exists $x_0\in K$ such that $\min\limits_{x\in K} f(x)=f(x_0)=\alpha>0$, so for any $x\in K$,
\[\dist (x,\bd (U_n))\geq \alpha.
\]
Let us introduce  the following  family of open cubes,
\[
\mathcal{Q}_\alpha(K)=\left\lbrace Q(x,l):\ x\in K, \ l\leq  \frac{\alpha}{\sqrt d}\right\rbrace.
\]
Notice that, if $Q\in\mathcal{Q}_\alpha (K)$, then $\overline{Q}\subset U_n$. Indeed, let $Q=Q(x,l)$ for some $(x_1,\dots,x_d)=x\in K$ and $l\leq\frac{\alpha}{\sqrt{d}}$. If $(y_1,\dots,y_d)=y\in \overline{Q}$, then for any $i=1,\dots, d$ we have
\[ 
|x_i-y_i|\leq\frac{l}{2},
\]
and so
\[|x-y|\leq\frac{l\sqrt{d}}{2}\leq \frac{\alpha}{2}.\]
For any $z\in \R^d\setminus U_n$ we have
\[
|z-x|\geq \dist (x,\bd (U_n))\geq \alpha,
\]
and thus
\[
|z-y|\geq ||z-x|-|x-y||\geq |z-x|-|x-y|\geq \alpha-\frac{\alpha}{2}=\frac{\alpha}{2}.
\]
Therefore $y\notin \R^d\setminus U_n$ and so $y\in U_n$, i.e. $\overline{Q}\subset U_n$. Since $\overline{Q}$ is closed and $U_n$ is open we have that $\overline{Q}\cap\bd (U_n)=\emptyset$. For each $x\in K$ denote
\[\mathcal{Q}_\alpha(x)=\left\lbrace Q(x,l): \ l\leq\frac{\alpha}{\sqrt d}\right\rbrace.\]
Clearly $\mathcal{Q}_\alpha (x)\subset \mathcal{Q}_\alpha (K)$ and $\mathcal{Q}_\alpha(x)$ forms a topological basis for neighborhoods of $x$ in the Euclidean topology of $\R^d$. Take now any open set $U$, such that $K\subset U$. For each $x\in K$, there exists $l_x>0$ with $Q(x,l_x)\in \mathcal{Q}_\alpha(x)$, such that $Q(x,l_x)\subset U$.

 Let $U$ be an open set, such that $K\subset U$ and $|U|<|K|+ \frac{1}{n+1}$. Then there exists a family $\left\lbrace Q(x,l_x)\right\rbrace_{x\in K}\subset \mathcal{Q}_\alpha (K)$ such that $\bigcup\limits_{x\in K}Q(x,l_x)\subset U $. By compactness of $K$, there exists $M_{n+1}\in\N$ and a finite subfamily $\left\lbrace Q(x_m,l_{x_m})\right\rbrace_{m=1}^{M_{n+1}}$ such that $K\subset\bigcup\limits_{m=1}^{M_{n+1}}Q(x_m,l_{x_m})$.
Setting $Q_{m,n+1}:=Q(x_m,l_{x_m})$,  $K\subset \bigcup\limits_{m=1}^{M_{n+1}}Q_{m,n+1}$ and $\left| \bigcup\limits_{m=1}^{M_{n+1}}Q_{m,n+1}\right|<|U|<|K|+\frac{1}{n+1}$.
Since $Q_{m,n+1}\in\mathcal{Q}_\alpha (K)$, we have
\[\overline{Q_{m,n+1}}\cap\bd (U_n)=\emptyset.\]
Define $U_{n+1}=\bigcup\limits_{m=1}^{M_{n+1}}Q_{m,n+1}$. Arguing as in case $n=1$, we get that
\[\left|\overline {U_{n+1}} \right|=\left| U_{n+1}\right|.\]
For every $m=1,\dots, M_{n+1}$, we have $\overline{Q_{m,n+1}}\subset U_n$. Therefore 
\[\overline{U_{n+1}}=\bigcup\limits_{m=1}^{M_{n+1}}\overline{Q_{m,n+1}}\subset U_n.\]
Recall that for any open set $U\subset\R^d$ with $|U|<\infty$ and any compact set $F\subset U$ we have that $|F|<|U|$. Hence,
\[\left|\overline{U_{n+1}}\right|<\left|U_n\right|.\]
By induction, we have constructed a sequence of open, bounded sets $\lbrace U_n\rbrace_{n=1}^\infty$ with desired properties.
\end{proof}
We will need further the following classical result.

\begin{theorem}
\label{Urynsohns lemma}\cite[Theorem 2.6.1]{LC}
Let $\Omega\subset \R^d$ be open. For every open $U\subset\Omega$  and compact $K\subset U$ there exists  $f\in C_C^\infty (\Omega)$, $f:\Omega\to [0,1]$ such that $f|_K\equiv 1$ and $\supp f\subset U$.
\end{theorem}

 Now we can prove that if $|\Sing\Phi|=0$ and for some compact set $K$ the function $\chi_{K}\in L^\Phi(\Omega)$ then $\chi_K$ can be approximated by elements of $C_C^\infty(\Omega)$.

\begin{theorem}\label{zwarte}
Let $\Phi$ be a MO function defined on an open $\Omega\subset \R^d$. If $|\Sing\Phi|=0$, then for any compact $K\subset\Omega$ such that $\chi_K\in E^\Phi(\Omega)$, there exists a sequence $\{ f_n\}_{n=1}^\infty\subset C_C^\infty(\Omega)$, such that $\lim\limits_{n\to\infty}\|  f_n-\chi_K\| _\Phi=0 $.
\end{theorem}\label{density}
\begin{proof}
Let $K\subset \Omega$ be a compact set such that $\chi_K\in E^\Phi(\Omega)$. Then there exists an open and bounded set $U$ such that $K\subset U$. Define
\[(\Sing\Phi)_U =\overline{\Sing\Phi\cap U}.\]
By Proposition \ref{dom}, $\Sing\Phi$ is closed. Hence,
\[(\Sing\Phi)_U =\overline{\Sing\Phi\cap U}=\Sing\Phi\cap\overline{U},\]
and $(\Sing\Phi)_U$ is compact.
By Lemma \ref{nested open sets}, there exists a sequence of open sets $\lbrace U_n\rbrace_{n=1}^\infty$, such that for every $n\in \N$,

\begin{enumerate}

    \item[$(1)$] $(\Sing\Phi)_U\subset U_n$,
    \item[$(2)$] $|\overline{U_n}|=|U_n|<\frac{1}{n}$,
    \item[$(3)$] $\overline{U_{n+1}}\subset U_n$ and $|U_{n+1}|<|U_n|$.
       
\end{enumerate}
     Take any $n\in\N$ and define $K_n=K\setminus U_n$. Clearly $K_n$ is compact and $K= (K\cap U_n) \cup K_n$. Setting $V_n=U\setminus\overline{U_{n+1}}$ we have that $K_n\subset V_n$ and $V_n$ is open.

     First we will show that $\chi_{V_n}\in E^\Phi(\Omega)$. Notice that \[V_n\subset\overline{U}\setminus U_{n+2}\subset \overline{U}\setminus(\overline{U}\cap\Sing\Phi)\subset \overline{U}\setminus\Sing\Phi.\]
   Hence, $V_n\cap\Sing\Phi\subset(\overline{U}\setminus U_{n+2})\cap\Sing\Phi =\emptyset$.
    By definition of $\Sing\Phi$, for each $x\in \overline{U}\setminus U_{n+2}$ there exists $r_x>0$ such that for all $\lambda>0$, \[\int\limits_{B(x,r_x)\cap\Omega}\Phi(y,\lambda)dy<\infty.\]
     On the other hand, since $V_n\subset \overline{U}\setminus U_{n+2} $ and $\overline{U}\setminus U_{n+2}$ is closed, so $\overline{V_n}\subset \overline{U}\setminus U_{n+2}$. Hence, for every $x\in\overline{V_n}$ here exists $r_x>0$ such that for all $\lambda>0$, 
     \[\int\limits_{B(x,r_x)\cap\Omega}\Phi(y,\lambda)dy<\infty.\]
  The family $\left\lbrace B(x,r_x)\right\rbrace_{x\in \overline{V_n}}$ is an open cover of $\overline{V_n}$, so by compactness of $\overline{V_n}$ there exists a finite subcover $\left\lbrace B(x_i,r_{x_i})\right\rbrace_{i=1}^k$ of $\overline{V_n}$. Then for any $\lambda>0$,
     \[\int\limits_\Omega,\Phi(x,\lambda\chi_{V_n}(x))dx\leq\frac{1}{k}\sum\limits_{i=1}^k\int\limits_{B(x_i,r_{x_i})\cap \Omega}\Phi(y,\lambda k)dy<\infty,\]
     and therefore $\chi_{V_n}\in E^\Phi(\Omega)$. Since $\chi_{K_n}\leq\chi_{V_n}$, we have $\chi_{K_n}\in E^\Phi(\Omega)$. Recall that $K_n=K\setminus U_n$ and $|U_n|<\frac{1}{n}$, hence $\lim\limits_{n\to\infty}\chi_K-\chi_{K_n}=0$ in measure. We can find a subsequence $\{n_k\}_{k=1}^\infty\subset \N$, such that $\lim\limits_{k\to\infty}\chi_{K_{n_k}}=\chi_K$ a.e.. Without loss of generality, assume that $\{\chi_{K_{n_k}}\}_{k=1}^\infty=\{\chi_{K_n}\}_{n=1}^\infty$. For any $\lambda>0$ we have $I_\Phi(\lambda\chi_K)<\infty$ and
     \[\Phi(x,\lambda(\chi_K(x)-\chi_{K_n}(x))\leq\Phi(x,\lambda_K(x))\text{ for a.e. }x\in\R^d.\]
Hence by the Lebesgue Dominated Convergence Theorem, 
 $\lim\limits_{n\to\infty}I_\Phi(\lambda(\chi_K-\chi_{K_n}))=0$
 so in view of Theorem \ref{podstawowe wlasnosci} (1),
 \begin{equation}\label{eq:1}
 \lim\limits_{n\to\infty}\| \chi_K-\chi_{K_n}\| _\Phi=0.
 \end{equation}
 For a fixed $n\in\N$, by Lemma \ref{nested open sets}, there exists a sequence of open sets $\left\lbrace W'_m\right \rbrace_{m=1}^\infty$ such that for every $m\in\mathbb N$,
     \begin{enumerate}
    \item[$(1')$]  $K_n\subset W'_m$,
    \item[$(2')$]$|\overline{W'_m}|=|W'_m|<|K_n|+\frac{1}{m}, $
    \item[$(3')$] $\overline {W'_{m+1}}\subset W'_m$.
 \end{enumerate}
Define now $W_m=W'_m\cap V_n$. 
For each $m\in \N$ we have $K_n\subset W_m$ and $W_m$ is open. Hence in view of Theorem \ref{Urynsohns lemma} there exists a function $f_{m,n}$, such that $f_{m,n}\in C_C^\infty(\Omega)$, $f_{m,n}:\Omega^d\to [0,1]$ and $f_{m,n}|_{K_n}\equiv 1$ and $\supp f_{m,n}\subset W_m$. Notice that 
$f_{m,n}\leq\chi_{W_m}$ a.e., and hence $\chi_{W_m}-f_{m,n}\leq\chi_{W_m\setminus K_n} $ a.e. For a fixed $\lambda>0$ and any $m\in\N$  by (2') we have
\begin{align*}
|\{x:\chi_{W_m}(x)-f_{m,n}(x)>\lambda\}|&\leq
|\{x:\chi_{W_m\setminus K_n}(x)>\lambda\}|\\
&\leq|W_m\setminus K_n|= |W_m|-|K_n|\\
&\leq |W'_m|-|K_n|<\frac{1}{m}.
\end{align*}
Therefore $\lim\limits_{m\to\infty}f_{m,n}= \chi_{K_n}$ in measure for every $n\in\N$. 
We can find a subsequence $\{m_k\}_{k=1}^\infty\subset \N$, such that  $\lim\limits_{k\to\infty}f_{m_k,n}=\chi_{K_n}$ a.e.
Without loss of generality, assume that $\{f_{m_k,n}\}_{k=1}^\infty=\{f_{m,n}\}_{m=1}^\infty$.
Since $f_{m,n}\leq\chi_{W_m}\leq\chi_{V_n} $ a.e. and $\chi_{V_n}\in E^\Phi(\Omega) $, so $f_{m,n}\in E^\Phi (\Omega)$. For any $\lambda>0$ we have $I_\Phi(\lambda\chi_{V_n})<\infty$ and
\[\Phi(x,\lambda (f_{m,n}(x)-\chi_{K_n}(x)))\leq\Phi (x, \lambda\chi_{V_n}(x))\text{ for a.e. }x\in\R^d.\]
By the Lebesgue Dominated Convergence Theorem  $\lim\limits_{m\to\infty}I_\Phi(\lambda (f_{m,n}-\chi_{K_n}))=0,$
 so, for every $n\in\N$ 
   \[\lim\limits_{m\to\infty}\|f_{m,n}-\chi_{K_n}\|_\Phi=0. \] 
Now, for every $n\in\N$ define $f_n=f_{m_n,n}$, where $m_n$ is the smallest integer such that 
 \begin{equation}\label{eq 3}
 \|f_{m_n,n}-\chi_{K_n}\|_\Phi<\frac{1}{n}.
  \end{equation}
Finally, by (\ref{eq:1}) and (\ref{eq 3}),
\[ \lim\limits_{n\to\infty}\| f_{n}-\chi_{K}\| _\Phi\leq \lim\limits_{n\to\infty}\| f_{n}-\chi_{K_n}\| _\Phi +\lim\limits_{n\to\infty}\| \chi_{K_n}-\chi_{K}\| _\Phi =0.\]
\end{proof}

Finally we can show that if $|\Sing\Phi|=0$, then functions from $C_C^\infty (\Omega)\cap E^\Phi (\Omega)$ are dense in $E^\Phi (\Omega)$.

\begin{theorem}\label{gestosc w E}
Let $\Omega\subset \R^d$ be open and $\Phi$ be a MO function defined on $\Omega$. If $|\Sing\Phi|=0$, then $C_C^\infty(\Omega)\cap E^\Phi (\Omega)$ is dense in $E^\Phi (\Omega)$. In other words $\overline{C_C^\infty(\Omega)\cap E^\Phi(\Omega)}=E^\Phi (\Omega)$. 
\end{theorem}
\begin{proof}
By Theorem \ref{podstawowe wlasnosci} (5), $E^\Phi(\Omega)\cap S(\Omega)$ is dense in $E^\Phi(\Omega)$. Every simple function is a finite linear combination of characteristic functions of measurable sets with finite measure. Hence, by linearity and in the view of Theorem \ref{zwarte} it suffices to show that any characteristic function  $\chi_A$ of a set of finite measure $A$ such that $\chi_A\in E^\Phi (\Omega)$ can be approximated by a characteristic function of a compact set.

For any such $A$, by inner regularity of the Lebesgue measure there exists a sequence of compact sets $K_n\subset A$, such that $|A\setminus K_n|<\frac{1}{n}$ for every $n\in\N$. For $\lambda>0$ we have
\[I_\Phi ( \lambda(\chi_A-\chi_{K_n}))=I_\Phi(\lambda (\chi_{A\setminus K_n})). \]
Clearly $\Phi(x,\lambda\chi_{A\setminus K_n}(x))\leq \Phi(x,\lambda\chi_{A}(x))$ for a.e. $x\in\Omega$. Since $I_\Phi(\lambda \chi_A)<\infty$, so
by the Lebesgue Dominated Theorem, we deduce that
\[\lim\limits_{n\to\infty}I_\Phi(\lambda(\chi_A-\chi_{K_n}))=0,\]
and so $\lim\limits_{n\to\infty}\| \chi_K-\chi_{K_n}\| _\Phi=0$.
\end{proof}

Since $E^\Phi (\Omega)= L^\Phi (\Omega)$ if and only if $\Phi$ satisfies $\Delta_2$  (Theorem \ref{podstawowe wlasnosci} (4)), the next result is an immediate consequence of Theorems  \ref{gestosc w E} and \ref{koniecznosc sing}.

\begin{corollary}\label{cor:gestosc w L^}

 Let $\Omega\subset \R^d$ be open and $\Phi$ be a MO function defined on $\Omega$ satisfying $\Delta_2$ condition. Then $C_C^\infty(\Omega)\cap L^\Phi (\Omega)$ is dense in $L^\Phi (\Omega)$ if and only if $|\Sing \Phi|=0$.
\end{corollary}

As a corollary we obtain extensions of Theorem 1 in \cite{YA} and Theorem 3.7.14 in \cite{HHK}. In the first theorem the authors assume local integrability of $\Phi$, while in the second one they assume  the so called (A0) condition, which  also implies  that $\Phi$ is locally integrable. In both cases $|\Sing \Phi|=0$.

An important class of $MO$ functions is a class of {\it double phase functionals} consisting of $\Phi: \Omega\times \mathbb{R}_+ \to \mathbb{R}_+$ such that 
\begin{equation}\label{eq:db}
\Phi(x,t) = t^{p(x)} + a(x) t^{r(x)},
\end{equation} 
where $a,p,r$ are measurable functions on $\Omega$,  $a(x) \ge 0$ a.e. and $1\le p(x) \le r(x) <\infty$ a.e.. Denote by $p^+$ and $p^-$,
\[
p^+ = \esssup_{x\in\Omega} p(x) \ \ \ \text{and} \ \ \ \ p^- = \essinf_{x\in\Omega} p(x),
\]
and $r^+$ and $r^-$ analogously as above.

\begin{corollary} Let $L^\Phi(\Omega)$ be the $MO$ space over an open set $\Omega\subset \mathbb{R}^d$, generated by double phase functional {\rm (\ref{eq:db})} with the assumption that the function $a(x)$ is essentially bounded on $\Omega$. 

If  {\rm(i)} $r^+ < \infty$, or {\rm (ii)} $p(x) < r(x)$ for a.a. $x\in {\Omega_1}$,
where $\Omega_1 = \supp a$,
and   $p^+< \infty$ and $r^+|_{\Omega_1}=\esssup_{x\in \Omega_1} r(x) < \infty$,
then  $C_C^\infty(\Omega)$ is dense in $L^\Phi (\Omega)$.

\end{corollary}

\begin{proof} Clearly $C_C^\infty(\Omega)\subset L^\Phi (\Omega)$. By Theorem 1.8 in \cite{Kamzyl}, if (i) or (ii) are satisfied then $\Phi$ satisfies $\Delta_2$ condition. Moreover in view of the assumptions on $a,p,r$, $|\Sing{\Phi}| =0$. We finish by applying Corollary \ref{cor:gestosc w L^}.
\end{proof}

Finally recall  that a given measurable function $1\le p(x)<\infty$ a.e. in $\Omega$, and 
 \[
 \Phi(x,t) = \frac{t^{p(x)}}{p(x)}, \ \ \ a.a. \  x\in\Omega, \ t\ge 0,
 \]
 the space $L^{p(\cdot)}(\Omega)= L^\Phi(\Omega)$ is called the {\it variable exponent Lebesgue space} \cite{CUF, DHHR, HHK, HH, zyl}.  The final well known  result is also a consequence of Corollary \ref{cor:gestosc w L^}.

\begin{corollary}\label{cor:power} \cite{DHHR}
 Let $L^{p(\cdot)}(\Omega)$ be the variable exponent Lebesgue space over open set $\Omega\subset \mathbb{R}^d$.
 If $p^+<\infty$ then $|\Sing \Phi|=0$.
 Consequently, $C_C^\infty(\Omega)$ is dense in $L^{p(\cdot)} (\Omega)$.
\end{corollary}


\end{document}